\documentclass[12pt,reqno]{amsart}
\usepackage{amsmath}
\usepackage{amssymb}
\usepackage{amstext}
\usepackage{a4wide}
\usepackage{graphicx}
\usepackage[numbers,sort&compress]{natbib}
\allowdisplaybreaks \numberwithin{equation}{section}
\usepackage{color}
\usepackage{cases}

\numberwithin{equation}{section}

\newtheorem{theorem}{Theorem}[section]
\newtheorem{proposition}[theorem]{Proposition}

\newtheorem{lemma}[theorem]{Lemma}
\newtheorem*{Yudovich's Theorem}{Yudovich's Theorem}

\theoremstyle{definition}

\theoremstyle{remark}
\newtheorem{remark}[theorem]{Remark}

\begin{document}

\title
[Nonlinear stability of sinusoidal Euler flows on a flat two-torus]{Nonlinear stability of sinusoidal Euler flows on a flat two-torus}

 \author{Guodong Wang,  Bijun Zuo}
\address{Institute for Advanced Study in Mathematics, Harbin Institute of Technology, Harbin 150001, P.R. China}
\email{wangguodong@hit.edu.cn}
\address{College of Mathematical Sciences, Harbin Engineering University, Harbin {\rm150001}, PR China}
\email{bjzuo@amss.ac.cn}


\begin{abstract}
Sinusoidal flows are an important class of explicit stationary solutions of the two-dimensional  incompressible Euler equations on a flat torus. For such flows, the steam functions are   eigenfunctions of  the negative Laplacian. In this paper, we prove that any sinusoidal flow related to some least eigenfunction  is, up to phase translations, nonlinearly stable  under $L^p$ norm of the vorticity for any  $1<p<+\infty$, which  improves a classical stability result by Arnold based on the energy-Casimir method.  The key point of the proof is to distinguish least eigenstates with fixed amplitude  from others  by using isovortical property of the Euler equations.

\end{abstract}

\maketitle
\section{Introduction and main result}
\subsection{Two-dimensional Euler equations on a flat torus}
Let $\mathbb T^{2}$ be a flat two-torus whose fundamental domain is
\[\mathbb T^{2}=
\left[0, 2{ \pi}{\nu_{1}}\right]\times\left[0, 2{ \pi}{\nu_{2}}\right],\]
 where $\nu_1,\nu_2$ are fixed positive constants.
 The motion of an ideal fluid of unit density on  $\mathbb T^{2}$ is described by  the following Euler equations:
\begin{equation}\label{euler}
\begin{cases}
\partial_t\mathbf v+(\mathbf v\cdot\nabla)\mathbf v=-\nabla P,&\mathbf x=(x_1,x_2)\in \mathbb T^2,\,\,t>0,\\
\nabla\cdot\mathbf v=0,
\end{cases}
\end{equation}
where $\mathbf v=(v_1,v_2)$ is the velocity field, and $P$ is the scalar pressure. The \emph{scalar vorticity} $\omega$ of the fluid is given by
\begin{equation}\label{deome}
\omega=\partial_{x_{1}}v_2-\partial_{x_{2}}v_1.
\end{equation}

For smooth solutions of \eqref{euler}, the following quantities are conserved for all time (see \cite{MB, MPu}):
 \begin{itemize}
 \item [(C1)] The total flux $\mathbf F$ of velocity,
 \begin{equation}\label{c1}
\mathbf F=\int_{\mathbb T^2}\mathbf v d\mathbf x,
\end{equation}
\item [(C2)] The kinetic energy $\mathcal E,$
\begin{equation}\label{vcon1}
\mathcal E=\frac{1}{2}\int_{\mathbb T^2}|\mathbf v |^2d\mathbf x,
\end{equation}
\item [(C3)] The distribution function of  vorticity $\mathsf d_{\omega(t,\cdot)}$,
\[\mathsf d_{\omega(t,\cdot)}(s)=|\{\mathbf x\in\mathbb T^2\mid \omega(t,\mathbf x)>s\}|,\quad s\in\mathbb R,\]
where $|\cdot|$ denotes the two-dimensional Lebesgue measure.
\end{itemize}
If we denote by $\mathcal R(f)$  the set of all equimeasurable rearrangements of some  $f\in L^1_{\rm loc}(\mathbb T^2)$, i.e.,
\[ 
\mathcal R(f)=\left\{g\in L^1_{\rm loc}(\mathbb T^2)\mid \mathsf d_g =\mathsf d_f  \right\},
\]
then the conservation of the distribution function of  vorticity can also be expressed as
\begin{equation}\label{iso0}
\omega(t,\cdot)\in\mathcal R(\omega(0,\cdot)),\quad\forall\,t\geq 0.
\end{equation}
As a consequence, the $L^p$  norm of vorticity is conserved for any $1\leq p\leq +\infty$.
In particular, the \emph{enstrophy} $Z$ of the fluid, defined by
\begin{equation}\label{deoz}
Z(\omega)=\frac{1}{2}\int_{\mathbb T^{2}}\omega^{2}d\mathbf x,
\end{equation}
is conserved.

Below we introduce the vorticity-stream   formulation of Euler equations \eqref{euler}. Define the \emph{normalized velocity} $\tilde{\mathbf v}$   as  
\begin{equation}\label{nv0}
\tilde{\mathbf v}=\mathbf v- \frac{1}{|\mathbb T^{2}|}\mathbf F,
\end{equation}
where $\mathbf F$ is the total flux given by \eqref{c1}. Note that $\mathbf F$ a constant vector not depending on the time variable. 
It is clear that $\tilde{\mathbf v}$  is divergence-free and has zero   integral  over $\mathbb T^2$. 
 By the discussion on  p. 50 of \cite{MB}, there is a function  $\tilde\psi:\mathbb T^2\to\mathbb R$, called the \emph{normalized stream function}, such that
\begin{equation}\label{nsf1}
\tilde{\mathbf v}=\nabla^\perp\tilde \psi.
\end{equation}
Here and henceforth, $\mathbf b^\perp=(b_{2},-b_{1})$ denotes the clockwise rotation through $\pi/2$ of some planar vector $\mathbf b=(b_{1},b_{2})$, and $\nabla^{\perp}f=(\nabla f)^{\perp}$ for some function $f$. Without loss of generality, by adding a suitable constant, we always assume that the normalized stream function has zero integral over $\mathbb T^2$.
Then $\tilde\psi$ satisfies
\begin{equation}\label{tpsi}
\begin{cases}
-\Delta\tilde\psi=\omega,&\mathbf x\in\mathbb T^{2},\\
\int_{\mathbb T^{2}}\tilde \psi d\mathbf x=0.
\end{cases}
\end{equation} 
By our Lemma \ref{bsl99} in Section 2, the  Poisson equation \eqref{tpsi} has a unique solution, denoted by  $\tilde\psi=K\omega$. Then  the velocity $\mathbf v$ can be determined by the vorticity $\omega$ and the total flux $\mathbf F$ as follows:
\begin{equation}\label{bsl}
\mathbf v=\nabla^\perp K\omega+\frac{1}{|\mathbb T^{2}|}\mathbf F,
\end{equation}
which is usually called the Biot-Savart law.
The \emph{stream function} $\psi$ of the fluid is defined by
 \begin{equation}
 \psi=K\omega-\frac{1}{|\mathbb T^{2}|}\mathbf F^{\perp}\cdot\mathbf x
 \end{equation}
 such that $\mathbf v=\nabla^{\perp}\psi.$
According to the above discussion,  at any time the state of the fluid can be described by  $\mathbf v$, or $\psi,$ or the pair $(\omega, \mathbf F)$.

The kinetic energy  can be expressed in terms of  $\omega$ and  $\mathbf F$   as follows:
\begin{equation}\label{deoe0}
\mathcal E =\frac{1}{2}\int_{\mathbb T^2}|\nabla K\omega|^2d\mathbf x+\frac{1}{2|\mathbb T^{2}|}|\mathbf F|^{2}.
\end{equation}
Define 
\begin{equation}\label{deoe}
E(\omega)=\frac{1}{2}\int_{\mathbb T^{2}}|\nabla K\omega|^{2}d\mathbf x=\frac{1}{2}\int_{\mathbb T^{2}}\omega K\omega d\mathbf x.
\end{equation}
Then by energy conservation, $E$ is also  conserved:
\begin{equation}\label{deoe6}
E(\omega(t,\cdot))=E(\omega(0,\cdot)),\quad\forall\, t\geq 0.
\end{equation}
 
\subsection{Sinusoidal flows and Arnold's stability result}
 Stationary solutions of the Euler equations \eqref{euler}  are characterized by having
  $\nabla \psi$ and   $\nabla\omega$ collinear. For this to hold, a sufficient condition  is that $\psi$ and $\omega$ satisfy the  relation
\begin{equation}\label{rel0}
\omega=f(\psi)
\end{equation}
for some function $f:\mathbb R\to\mathbb R.$  In particular, if
$f$ is  linear,
then \eqref{rel0} becomes the eigenvalue problem of $-\Delta$ on $\mathbb T^2$:
\begin{equation}\label{rel00}
-\Delta \psi=\lambda\psi.
\end{equation}
For \eqref{rel00}, 
 any eigenvalue $\lambda$   has the form (see  \cite{Ti}, Chapter 1)
\begin{equation}\label{rel1}
\lambda=\left(\frac{k_1}{\nu_1}\right)^2+\left(\frac{k_2}{\nu_2}\right)^2
\end{equation}
 for some integers $k_1,k_2$   such that  $k_1^2+k_2^2\neq0$, and the corresponding eigenspace is  spanned by
 \begin{align*}
 \sin\left(\frac{j_1}{\nu_1}x_1 \right)\sin\left( \frac{j_2}{\nu_2}x_2 \right),\quad\sin\left(\frac{j_1}{\nu_1}x_1 \right)\cos\left( \frac{j_2}{\nu_2}x_2\right),\\
 \cos\left(\frac{j_1}{\nu_1}x_1 \right)\sin\left( \frac{j_2}{\nu_2}x_2 \right),\quad\cos\left(\frac{j_1}{\nu_1}x_1 \right)\cos\left( \frac{j_2}{\nu_2}x_2 \right),
\end{align*}
for all integers $j_1, j_2$  such that  
\begin{equation}\label{rel18}
\left(\frac{j_1}{\nu_1}\right)^2+\left(\frac{j_2}{\nu_2}\right)^2=\lambda.
\end{equation}
A \emph{sinusoidal flow}, or an \emph{eigenstate}, is an Euler flow whose stream function is some eigenfunction of $-\Delta$  on $\mathbb T^{2}$.
For a sinusoidal flow related to some eigenvalue $\lambda,$ the stream function can  be written in the following form:
 \begin{equation*} 
 \sum_{(j_{1},j_{2})\in J_{\lambda}}A_{j_{1},j_{2}}\sin\left(\frac{j_{1}}{\nu_{1}}x_{1}+\frac{j_{2}}{\nu_{2}}x_{2}+\alpha_{j_{1},j_{2}}\right)+\sum_{(j_{1},j_{2})\in J_{\lambda}}B_{j_{1},j_{2}}\sin\left(\frac{j_{1}}{\nu_{1}}x_{1}-\frac{j_{2}}{\nu_{2}}x_{2}+\beta_{j_{1},j_{2}}\right),
 \end{equation*}
where $A_{j_{1},j_{2}}, B_{j_{1},j_{2}}\geq 0, $  $\alpha_{j_{1},j_{2}}, \beta_{j_{1},j_{2}}\in\mathbb R,$  and  $J_{\lambda}$ is defined by
\begin{equation}\label{j1j2}
J_{\lambda}=\left\{(j_{1},j_{2})\in\mathbb Z^{2}\mid \left(\frac{j_1}{\nu_1}\right)^2+\left(\frac{j_2}{\nu_2}\right)^2=\lambda\right\}.
\end{equation}
 Note that for any sinusoidal flow, the total flux of velocity is $\mathbf 0$, hence the normalized velocity is equal to the velocity, and the normalized stream function is equal to the stream function.

The stability  of sinusoidal flows  is a fundamental problem in fluid dynamics and has been  extensively studied in the literature.  For the linear theory, the results are quite rich, although many open questions still remain. See \cite{BN,BFY,DW,FS,LY,MSi} and the references therein.  As to nonlinear stability, the  first rigorous  result was obtained by Arnold. In the 1960s, Arnold \cite{A1,A2} proved two nonlinear stability  criteria for plane ideal flows, now usually referred to as Arnold's first and second stability theorems. See also \cite{WGu1,WG0,WG} for some of their extensions.  As a straightforward application of Arnold's second stability theorem,  partial nonlinear stability    can be proved for sinusoidal flows related to least eigenfunctions of $-\Delta$ on $\mathbb T^2.$

To state Arnold's result, we briefly analyze the least eigenvalue $\lambda_1$ of $-\Delta$ on $\mathbb T^2.$
 By \eqref{rel1}, there are three cases:
\begin{itemize}
\item[(i)] If $\mathbb T^{2}$ is a short torus, i.e., $\nu_1<\nu_2,$ then $\lambda_1=\nu_2^{-2}$. In this case, $J_{\lambda_{1}}=\left\{(0,1), (0,-1)\right\}$,  thus any least eigenfunction $\psi_1$ takes the form
\begin{equation}\label{x1}
\psi_1(x_1,x_2)=A\sin\left(  \frac{x_2}{\nu_2} +\alpha\right),
\end{equation}
where $A\geq 0$ is called the amplitude,  and $\alpha\in\mathbb R$ is called the phase parameter.
\item[(ii)] If $\mathbb T^{2}$ is a long torus, i.e., $\nu_1>\nu_2,$ then $\lambda_1=\nu_1^{-2}$. In this case, $J_{\lambda_{1}}=\left\{(1,0), (-1,0)\right\}$,   hence any least eigenfunction  $\psi_{1}$  takes  the form
\begin{equation}\label{x2}
\psi_1(x_1,x_2)=A\sin\left(  \frac{x_1}{\nu_2} +\alpha\right) 
\end{equation}
for some $A\geq 0$ and $\alpha\in\mathbb R$.
\item[(iii)] If $\mathbb T^{2}$ is a square torus, i.e., $\nu_1=\nu_2=\nu,$ then $\lambda_1=\nu^{-2}$. In this case, $J_{\lambda_{1}}=\left\{(0,1), (0,-1),(1,0),(-1,0)\right\}$,  hence any least eigenfunction  $\psi_{1}$  takes  the form
\begin{equation}\label{x3}
\psi_1(x_1,x_2)=A\sin\left(  \frac{x_1}{\nu } +\alpha\right)+B\sin\left(  \frac{x_2}{\nu } +\beta\right)
\end{equation}
for some $A, B\geq 0$ and $\alpha, \beta\in\mathbb R$.
\end{itemize}
For convenience of subsequent presentation,   we call a sinusoidal flow of the form \eqref{x1} an $x_2$-mode, and a sinusoidal flow of the form \eqref{x2} an $x_1$-mode.
It is clear that all  $x_1$-modes, as well as all $x_2$-modes,   form a two-dimensional vector space, and all sinusoidal flows of the form \eqref{x3} form a four-dimensional vector space.

 Since it is more convenient to express Arnold's  result in terms of vorticity,
we denote by $\mathcal V_i$ the set of  vorticity functions of all $x_i$-modes,  i.e.,
\begin{equation}\label{dev01}
\mathcal V_i=\left\{A\sin\left(\frac{x_i}{\nu_i}+\alpha\right) \mid A\geq 0,\,\alpha\in\mathbb R \right\},\quad i=1,2.
\end{equation}
If $\nu_1=\nu_2=\nu$, we denote by $\mathcal V$ the set of   vorticity functions of all sinusoidal flows of the form  \eqref{x3}, i.e.,
\begin{equation}\label{dev02}
\mathcal V =\left\{A\sin\left(\frac{x_1}{\nu}+\alpha\right)+B\sin\left(\frac{x_2}{\nu}+\beta\right)\mid A,B\geq0, \,\alpha,\beta\in\mathbb R \right\}.
\end{equation}

Arnold's result can be stated as follows.
\begin{theorem}[Arnold, \cite{A1,A2}]\label{thm0}
Let $\mathcal V_1,\mathcal V_2,\mathcal V$ be defined by \eqref{dev01}, \eqref{dev02}.  Then the following assertions hold:
\begin{itemize}
\item [(i)] If  $\nu _1<\nu_2$, then $\mathcal V_2$ is nonlinearly stable in $L^2$ norm, i.e.,  for any $\varepsilon>0,$ there exists some $\delta>0$, such that for any smooth   Euler flow on $\mathbb T^2$ with vorticity $\omega$, we have that
\begin{equation}\label{sfg6}
\min_{v\in\mathcal V_2}\|\omega(0,\cdot)-v\|_{L^2 (\mathbb T^2)}<\delta\Longrightarrow \min_{v\in\mathcal V_2}\|\omega(t,\cdot)-v\|_{L^2(\mathbb T^2)}<\varepsilon \quad\forall\,t>0.
\end{equation}
 \item[(ii)] If  $\nu_1>\nu_2$, then $\mathcal V_1$ is nonlinearly  stable in $L^2$ norm, i.e., \eqref{sfg6} holds with $\mathcal V_2$   replaced by $\mathcal V_1$.
   \item[(iii)] If  $\nu_1=\nu_2=\nu$,  then $\mathcal V$ is nonlinearly  stable in $L^2$ norm, i.e., \eqref{sfg6} holds with $\mathcal V_2$  replaced by $\mathcal V$.
  \end{itemize}
 \end{theorem}
\begin{remark}
By rotational symmetry,  items   (i)  and  (ii) in Theorem \ref{thm0}  in fact tell the same thing.
\end{remark}
\begin{remark}
If $\nu _1<\nu_2$ (or $\nu _1>\nu_2$), then any $x_1$-mode ($x_2$-mode, accordingly) is known to be linearly unstable in a certain sense. See \cite{BFY,FS,MSi} for example.
\end{remark}

 Theorem \ref{thm0} can also be found in \cite{AK}, p. 98,  or  \cite{MPu}, p. 111. For the reader's convenience,  we provide a detailed proof following Arnold's original idea in Appendix \ref{appe1}.

By Theorem \ref{thm0},  any least eigenstate ($x_{2}$-mode  on  a short torus, or  $x_{1}$-mode on a long torus, or sinusoidal  flow  of the form \eqref{x3} on a square torus),
 is nonlinearly stable  \emph{up to phase translations and amplitude scalings}. 
For example, given an $x_{2}$-mode with vorticity $A_0\sin(\nu_2^{-1}x_2+\alpha_0)$ on a  short torus, if  a smooth Euler flow    is  ``close"  to this $x_2$-mode at initial time,  then at any $t>0$ the evolved flow   is ``close" to some $x_2$-mode with vorticity $A_t\sin(\nu_2^{-1}x_2+\alpha_t)$.   Here  ``closeness"
is measured in terms of  $L^2$ norm of the vorticity. Since $A_t$ and $\alpha_t$ may vary with time, it is not clear whether the $x_{2}$-mode with vorticity $A_0\sin(\nu_2^{-1}x_2+\alpha_0)$  is nonlinearly stable.

\subsection{Main result}

The nonlinear stability of a  single sinusoidal  flow was listed as an open problem on 
 p. 112 of  Marchioro-Pulvirenti's book \cite{MPu}.   Arnold's method can not handle this problem since it  only involves   energy and enstrophy conservations (see Appendix \ref{appe1}), which is not enough to distinguish different sinusoidal flows.
In \cite{WS},  Wirosoetisno-Shepherd employed high-order (cubic, quartic and quintic) Casimirs  to bound the variation of the amplitudes $A,B$ in the case of a square torus. As a consequence, they obtained the  nonlinear stability of a single sinusoidal flow up to  phase translations. However, since the bound therein depends on high-order Casimirs of the initial state,  rigorous nonlinear stability (even up to  phase translations) remains unclear.

Our purpose in this paper is to give an extension of Theorem \ref{thm0} and implement the idea in \cite{WS} rigorously. To state our result, for fixed $A\geq  0$, define
\begin{equation}\label{nce1}
\mathcal V_{A,i}=\left\{A\sin\left(\frac{x_i}{\nu_i}+\alpha\right) \mid \alpha\in\mathbb R \right\}.
\end{equation}
If $\nu_1=\nu_2=\nu,$ for fixed $A,B\geq 0, $ define
\begin{equation}\label{nce2}
\mathcal V_{A,B}=\left\{A\sin\left(\frac{x_1}{\nu}+\alpha\right)+B\sin\left(\frac{x_2}{\nu}+\beta\right) \mid \alpha,\beta\in\mathbb R \right\}.
\end{equation}
It is easy to see that
\[\mathcal V_i=\underset{ A>0}\cup\mathcal V_{A,i},\,\, i=1,2,\quad \mathcal V=\underset{A,B\geq0}\cup \mathcal V_{A,B}.\]

Our main result in this paper is the following theorem.
\begin{theorem}\label{thm1}
Let   $1<p<+\infty$ be fixed.   Then the following assertions hold:
\begin{itemize}
\item[(i)] If  $\nu_1< \nu_2$, then for any $A\geq 0$,    $\mathcal V_{A,2}$ is nonlinearly stable in  $L^p$ norm, i.e.,  for any $\varepsilon>0,$ there exists some $\delta>0$, such that for any  smooth Euler flow on $\mathbb T^{2} $ with vorticity $\omega$, we have that
\begin{equation}\label{sfg1}
\min_{v\in\mathcal V_{A,2}}\|\omega(0,\cdot)-v\|_{L^p (\mathbb T^2)}<\delta
\Longrightarrow \min_{v\in\mathcal V_{A,2}}\|\omega(t,\cdot)-v\|_{L^p(\mathbb T^2)}<\varepsilon\quad\forall\,t>0.
\end{equation}
\item[(ii)] If  $\nu_1> \nu_2$, then   for any $A\geq 0$,   $\mathcal V_{A,1}$ is nonlinearly stable in $L^p$ norm, i.e., \eqref{sfg1} holds with $\mathcal V_{A,2}$   replaced by $\mathcal V_{A,1}$.
\item[(iii)] If $\nu_1=\nu_2=\nu$, then  for any $A,B\geq 0$,   $\mathcal V_{A,B}$ is nonlinearly stable in  $L^p$ norm, i.e., \eqref{sfg1} holds with $\mathcal V_{A,2}$   replaced by $\mathcal V_{A,B}$.
\end{itemize}
\end{theorem}

 \begin{remark}
In Theorem \ref{thm1}, to avoid some technical (but not essential) difficulties and illustrate the main idea clearly, we assume that the perturbed flows are smooth.  However, by checking the proof carefully,  Theorem \ref{thm1} actually holds for a large class of less regular perturbations such that (i) the quantities (C1)-(C3) are conserved; (ii) the vorticity is continuous in $L^p(\mathbb T^2)$ with respect to the time variable; (iii) a ``follower" to the perturbed vorticity as in Section 5 exists.
 \end{remark}
 
 \begin{remark}\label{wen1}
By Lemma \ref{den1} in Section 2, the nonlinear stabilities in Theorem \ref{thm1} can also be measured in terms of $W^{1,p}$ norm of the normalized velocity, or $W^{2,p}$ norm of the normalized stream function for any $1<p<+\infty$. For example, Theorem \ref{thm1}(i) can be equivalently stated as follows.
 \begin{itemize}
 \item[(1)] If  $\nu_1< \nu_2$, then for any $A\geq 0$, it holds that:  for any $\varepsilon>0,$ there exists some $\delta>0$, such that for any  smooth Euler flow on $\mathbb T^{2} $ with normalized velocity  $\tilde{\mathbf v}$, we have that
\begin{equation*}
\min_{\mathbf u\in\mathsf V_{A,2}}\|\tilde{\mathbf v}(0,\cdot)-\mathbf u\|_{W^{1,p} (\mathbb T^2)}<\delta
\Longrightarrow \min_{\mathbf u\in\mathsf V_{A,2}}\|\tilde{\mathbf v}(t,\cdot) -\mathbf u\|_{W^{2,p}(\mathbb T^2)}<\varepsilon\quad\forall\,t>0,
\end{equation*}
where   $\mathsf V_{A,2}$ is the set of  velocities (also the set of normalized velocities) related to $\mathcal V_{A,2},$ given by
\[\mathsf V_{A,2}=\left\{\left(A\nu_{2}\cos\left(\nu^{-1}_{2} {x_2} +\alpha\right),0\right)\mid \alpha\in\mathbb R \right\}.\]
 \item[(2)] If  $\nu_1< \nu_2$, then for any $A\geq 0$,  it holds that:  for any $\varepsilon>0,$ there exists some $\delta>0$, such that for any  smooth Euler flow on $\mathbb T^{2} $ with normalized stream function $\tilde \psi$, we have that
\begin{equation*}
\min_{\varphi\in\mathsf S_{A,2}}\|\tilde\psi(0,\cdot)-\varphi\|_{W^{2,p} (\mathbb T^2)}<\delta
\Longrightarrow \min_{\varphi\in\mathsf S_{A,2}}\|\tilde\psi(t,\cdot)-\varphi\|_{W^{2,p} (\mathbb T^2)}<\varepsilon\quad\forall\,t>0,
\end{equation*}
where   $\mathsf S_{A,2}$ is the set of  stream functions  (also the set of normalized stream functions) related to $\mathcal V_{A,2},$ given by
\[\mathsf S_{A,2}=\left\{A\nu_{2}^{2}\sin\left(\nu_{2}^{-1}{x_2} +\alpha\right) \mid \alpha\in\mathbb R \right\}.\]

 \end{itemize}
 
 \end{remark}

By Theorem \ref{thm1},  any eigenstate on $\mathbb T^{2}$ is nonlinearly stable \emph{only up to phase translations}, which  noticeably improves Theorem \ref{thm0}. Moreover, the stabilities  in Theorem \ref{thm1} are measured in terms of $L^p$ norm of the vorticity for any $1<p<+\infty,$ which are also more general than those in Theorem \ref{thm0}.

To prove Theorem \ref{thm1}, we  use a variational approach in combination with a compactness argument, which is very different from the classical energy-Casimir method used by
 Arnold in \cite{A1,A2} and  Wirosoetisno-Shepherd in \cite{WS}.
Our method  consists of three  ingredients:  a suitable variational characterization for the sinusoidal flows under consideration, a compactness argument, and proper use of flow invariants.
These three ingredients are also essential in the nonlinear stability analysis of many other stationary Euler flows. See \cite{Abe,BG,Bjde,Bcmp,CWCV,CWN,CD,WGu2} for example.
The   variational characterizations, which states that the sinusoidal flows in  Theorem \ref{thm1} are exactly the set of  maximizers of the conserved functional $E$ relative to all isovortical flows to them, are the most important step in the whole proof.  The advantage of such variational characterizations is that we are able to distinguish the sinusoidal flows  with vorticity in $\mathcal V_{A,1},\mathcal V_{A,2} $ or $\mathcal V$  from other least eigenstates.  

The nonlinear stability of a single sinusoidal flow still remains open. In our method, we only use energy and vorticity conservations,
which are not enough to differentiate one flow from another within the set of flows with vorticity in $\mathcal V_{A,1}, \mathcal V_{A,2}$ or $\mathcal V_{A,B}$. Hence to improve  Theorem \ref{thm1} further, new flow invariants are needed.   This is an interesting further work.

This paper is organized as follows. In Section 2, we give some preliminary materials for later use. In  Section 3, we establish variational principles for the sinusoidal flows  under consideration. In Section 4, we prove   compactness  related to the variational principles established in Section 3. In Section 5, we give the proof of Theorem \ref{thm1}.

 \section{Preliminaries}

\subsection{Definitions, notation  and  basic facts}

 \begin{itemize}
  \item For $1\leq p\leq +\infty$, denote by $L^{p}(\mathbb T^2)$ the set of all  $p$-th power integrable (essentially bounded if $p=+\infty$) real-valued   functions on $\mathbb T^2$.  The norm of $L^{p}(\mathbb T^2)$ is denoted by $\|\cdot\|_{L^{p}(\mathbb T^{2})}.$
  \item For $1\leq p\leq +\infty$ and $k\in \mathbb Z^{+}$, where $\mathbb Z^+$ is the set of positive integers, $W^{k,p}(\mathbb T^{2})$ denotes the set of all real-valued functions whose  weak derivatives up to order $k$ are $p$-th power integrable (essentially bounded if $p=+\infty$). The norm of $W^{k,p}(\mathbb T^2)$ is denoted by $\|\cdot\|_{W^{k,p}(\mathbb T^{2})}.$ Note that if we consider $D=(0, 2\pi \nu_{1} )\times (0, 2\pi \nu_{2}),$  a domain of $\mathbb R^{2},$ then $W^{k,p}(\mathbb T^2)$ can be regarded as a closed subspace of $W^{k,p}(D)$, and $f\in W^{k,p}(\mathbb T^2)$ if and only if $f\in W^{k,p}(D)$ and satisfies
  \[  f(0,x_{2})=f(2\pi\nu_{1},x_{2}),\quad \forall\,0\leq x_{2}\leq  2\pi \nu_{2},\]
  \[f( x_{1},0)=f( x_{1},2\pi\nu_{2}),\quad\forall\,0\leq x_{1}\leq 2\pi\nu_{1}\]
     in the sense of traces.
   \item For $1\leq p\leq +\infty$ and $k\in \mathbb Z^{+}$, denote
   \[\mathring {L}^p(\mathbb T^2)=\left\{f\in L^{p}(\mathbb T^{2})\mid \int_{\mathbb T^{2}}fd\mathbf x=0\right\},\]
   \[ \mathring {W}^{k,p}(\mathbb T^2)=\left\{f\in W^{k,p}(\mathbb T^{2})\mid \int_{\mathbb T^{2}}fd\mathbf x=0\right\}.\]
 It is clear that $\mathring {L}^p(\mathbb T^2)$  is a closed subspace of ${L}^p(\mathbb T^2)$, and  $\mathring {W}^{k,p}(\mathbb T^2)$  is a closed subspace of   ${W}^{k,p}(\mathbb T^2)$.

    \item Denote by $L^2(\mathbb T^2;\mathbb C)$  the set of all  $p$-th power integrable   complex-valued  functions on $\mathbb T^2$ endowed with the following inner product:
  \[<f,g>=\int_{\mathbb T^2}f(\mathbf x)\overline{g(\mathbf x)}d\mathbf x,  \quad\forall\,f,g\in L^2(\mathbb T^2;\mathbb C),\]
  where $\overline{g(\mathbf x)}$ is the complex conjugate of $g(\mathbf x).$

  \item   Denote by $\mathbb Z^2$  the set of all points in $\mathbb R^2$ with integer coordinates.  For $\mathbf k=(k_1,k_2)\in\mathbb Z^2$, define
 \[\zeta_{\mathbf k}(\mathbf x)=\frac{1}{  \sqrt{4\pi^2\nu_1\nu_2}}e^{ i\left( \frac{k_1}{\nu_1}  x_1+\frac{k_2}{\nu_2}  x_2\right)},\quad\mathbf x=(x_1,x_2)\in \mathbb T^2.\]
 Then $\{\zeta_{\mathbf k}\}_{\mathbf k\in\mathbb Z^2}$ is an orthonormal basis of $L^2(\mathbb T^2;\mathbb C)$ (see \cite{Gra}, p. 186). For $i,j=1,2,$ it is easy to check that
  \begin{equation}\label{qiud}
  \partial_{x_{i}}\zeta_{\mathbf k}=  i  \left(\frac{k_i}{\nu_i}\right) \zeta_{\mathbf k},\quad \partial_{x_{i}x_{j}}\zeta_{\mathbf k}=-\left(\frac{k_i}{\nu_i}\right)\left(\frac{k_j}{\nu_j}\right)\zeta_{\mathbf k}.
  \end{equation}
  In particular,
    \begin{equation}\label{qiud8}
-\Delta\zeta_{\mathbf k}=\left[\left(\frac{k_1}{\nu_1}\right)^2+\left(\frac{k_2}{\nu_2}\right)^2\right]\zeta_{\mathbf k}.
  \end{equation}
 \item  For  $f\in L^1(\mathbb T^2;\mathbb C)$ and $\mathbf k\in\mathbb Z^2$, denote by $\hat f_{\mathbf k}$ the $\mathbf k$-th Fourier coefficient of $f$, i.e.,
  \[\hat f_{\mathbf k}=<f,\zeta_{\mathbf k}>=\int_{\mathbb T^2}f(\mathbf x)\overline{\zeta_{\mathbf k}(\mathbf x)}d\mathbf x.\]
 The Fourier series of $f$ is  then
 \[f\thicksim\sum_{\mathbf k\in\mathbb Z^2}\hat f_{\mathbf k}\zeta_{\mathbf k}.\]
 \item For any $f\in L^1(\mathbb T^2;\mathbb C)$ and  $N\in\mathbb Z^{+}$, denote by $f_N$ the $N$-th \emph{circular} partial sum of the Fourier series of $f$, i.e.,
 \begin{equation}\label{fn}
 f_N =\sum_{\mathbf k\in \mathbb Z^2,   |\mathbf k|_\infty\leq N}\hat f_{\mathbf k}\zeta_{\mathbf k},\quad\mbox{where }\,\, |\mathbf k|_\infty:=\max\{|k_1|,|k_2|\}.
 \end{equation}
Note that if $f$ is real-valued, then $\hat f_{-\mathbf k}=\overline{f_{\mathbf k}}$ for any $\mathbf k\in\mathbb Z^2$, thus $f_N$ is also real-valued for any  $N\in\mathbb Z^{+}$. Also note that for fixed $1<p<+\infty$, if  $f\in L^p(\mathbb T^2;\mathbb C),$ then  $f_N\to f$ in $L^p(\mathbb T^2)$ as $N\to+\infty$ (see \cite{Gra}, Theorem 4.1.8).
 \end{itemize}

\subsection{Poisson equation on a flat torus}\label{app0}
In this subsection, we study the following Poisson equation:
   \begin{equation}\label{ppo}
   \begin{cases}
  -\Delta u=f, &\mathbf x\in\mathbb T^2,\\
  u\in \mathring{W}^{2,p} (\mathbb T^2),
  \end{cases}
 \end{equation}
 where  $f\in \mathring{L}^p(\mathbb T^2)$, $1<p<+\infty$.

 \begin{lemma}\label{bsl99}
Let  $1<p<+\infty$. Then for any $f\in \mathring{L}^p(\mathbb T^2)$,
 there exists a unique solution  $u $  to the Poisson equation \eqref{ppo}.
 Moreover, the following estimate holds:
 \begin{equation}\label{bddo}
 \|u\|_{W^{2,p}(\mathbb T^2)}\leq C\| f\|_{L^p(\mathbb T^2)},
 \end{equation}
where $C>0$ depends only on $\nu_1,\nu_2$ and $p.$
 \end{lemma}

\begin{proof}
First we prove existence.  Consider the following approximate equation:
 \begin{equation}\label{ppo1}
 \begin{cases}
 -\Delta u_N=f_N,&\mathbf x\in \mathbb T^2,\\
 u_{N}\in\mathring{W}^{2,p}(\mathbb T^{2}),\\
 \end{cases}
 \end{equation}
where $f_N$ is $N$-th  circular  partial sum of the Fourier series of $f$, defined by \eqref{fn}.  Since $f\in \mathring{L}^{p}(\mathbb T^{2})$, we have  $\hat f_{\mathbf 0}=0$.
Then it is easy to check that \eqref{ppo1} admits an explicit solution:
\[u_N=\sum_{\mathbf k\in \mathbb Z^2,   0<|\mathbf k|_\infty\leq N}\frac{\hat f_{\mathbf k}}{  \left( \frac{k_1}{\nu_1}\right)^2+\left( \frac{k_2}{\nu_2}\right)^2 }\zeta_{\mathbf k}.\]
Moreover, we have the following uniform estimate for  $u_{N}$ (see \cite{CZ}, Theorem 10):
\begin{equation}\label{gla1}
\|\partial_{x_{i}x_{j}}u_N\|_{L^p(\mathbb T^2)}\leq C\|f_N\|_{L^p(\mathbb T^2)}, \quad \forall\,i,j=1,2,
\end{equation}
where $C>0$ depends only on $\nu_1,\nu_2$ and $p.$ Applying   the Poincar\'e inequality (notice that $u_N\in \mathring L^p(\mathbb T^2)$ and $\partial_iu_N\in \mathring L^p(\mathbb T^2)$, $ i=1,2$), we further  have that
\begin{equation}\label{kja1}
\|u_N\|_{W^{2,p}(\mathbb T^2)}\leq C\|f_N\|_{L^p(\mathbb T^2)}.
\end{equation}
Similarly, for any  $N_{1},N_{2}\in\mathbb Z^+,$
\begin{equation}\label{kja2}
\|u_{N_1}-u_{N_2}\|_{W^{2,p}(\mathbb T^2)}\leq C\|f_{N_1}-f_{N_2}\|_{L^p(\mathbb T^2)}.
\end{equation}
From \eqref{kja2}, taking into account the fact that $f_{N}\to f$ in $L^{p}(\mathbb T^{2})$ as $N\to+\infty$, we see that $\{u_N\}$ is a Cauchy sequence in $\mathring W^{2,p}(\mathbb T^2),$
 thus $u_N$ converges to some $u$ in $\mathring W^{2,p}(\mathbb T^2)$ as $N\to+\infty$.  It is clear that  $-\Delta u=f$ a.e. $\mathbf x\in\mathbb T^2$, hence  $u$  solves \eqref{ppo}. Moreover, passing to the limit $N\to+\infty$ in \eqref{kja1} gives
\[\|u\|_{W^{2,p}(\mathbb T^2)}\leq C\|f\|_{L^p(\mathbb T^2)}.\]

Next we prove uniqueness. Suppose \eqref{ppo} has two solutions, say $u_1 $ and $u_2$. Then
\[-\Delta (u_1-u_2)=0,\quad\mathbf x\in\mathbb T^{2}.\]
By integration by parts,
\[\int_{\mathbb T^2}|\nabla(u_1-u_2)|^2d\mathbf x=0,\]
which implies that $u_1=u_2+c$ for some constant $c$. Taking into account the fact that  $u_{1},u_{2}\in \mathring{W}^{2,p}(\mathbb T^{2})$,
we obtain $u_1\equiv u_2$.
\end{proof}

By Lemma \ref{bsl99},  the negative Laplacian on $\mathbb T^{2}$ has an inverse, denoted by $K$. The estimate \eqref{bddo} indicates that $K$ is a bounded operator from $\mathring L^{p}(\mathbb T^{2})$ to $\mathring W^{2,p}(\mathbb T^{2})$.

The following lemma, asserting that $K$ is symmetric and positive definite, is crucial to the proof of Proposition \ref{compact} in Section 4.

\begin{lemma}\label{sppp}
Let $1<p<+\infty$ be fixed. Then
\begin{itemize}
\item[(i)] for any $f,g\in\mathring{L}^{p}(\mathbb T^{2}),$ it holds that
\begin{equation}\label{symm}
\int_{\mathbb T^{2}}fKg d\mathbf x=\int_{\mathbb T^{2}}gKf d\mathbf x;
\end{equation}
\item [(ii)]for any $f\in\mathring{L}^{p}(\mathbb T^{2}), $ it holds that
\begin{equation}\label{pode}
\int_{\mathbb T^{2}}fKf d\mathbf x\geq 0,
\end{equation}
and the  equality holds  if and only if $f\equiv 0.$
\end{itemize}
\end{lemma}
\begin{proof}
First we prove (i). Denote $u=Kf$, $v=Kg$. By integration by parts,
\[\int_{\mathbb T^{2}}fKg d\mathbf x=\int_{\mathbb T^{2}}(-\Delta u)v d\mathbf x=\int_{\mathbb T^{2}}\nabla u\cdot\nabla vd\mathbf x.\]
Similarly,
\[\int_{\mathbb T^{2}}gKf d\mathbf x=\int_{\mathbb T^{2}}(-\Delta v)u d\mathbf x=\int_{\mathbb T^{2}}\nabla v\cdot\nabla ud\mathbf x.\]
Hence \eqref{symm} holds.

Next we prove (ii). Still denote $u=Kf$. Then integration by parts gives
\[\int_{\mathbb T^{2}}fKf d\mathbf x=\int_{\mathbb T^{2}}(-\Delta u)u d\mathbf x=\int_{\mathbb T^{2}}|\nabla u|^{2}d\mathbf x.\]
Hence \eqref{pode} holds. Moreover, \eqref{pode} is an equality if and only if $u\equiv 0$, which is equivalent to $f\equiv 0.$
\end{proof}

Denote by $\mathcal V_1^\perp,\mathcal V_{2}^{\perp},\mathcal V^{\perp}$   the orthogonal complements of $\mathcal V_1,\mathcal V_{2},\mathcal V$   in $\mathring L^2(\mathbb T^2),$ respectively.  The following lemma will also be needed in subsequent sections.

 \begin{lemma}\label{fam}
For $i=1,2$, it holds that
 \begin{equation}\label{fam1}
 \int_{\mathbb T^{2}}\nabla Kg\cdot\nabla Kh d\mathbf x=0,\quad\forall\,g\in\mathcal V_{i},\,h\in\mathcal V_{i}^{\perp}.
 \end{equation}
  If additionally  $\nu_1=\nu_2=\nu$, then
 \begin{equation}\label{fam2}
 \int_{\mathbb T^{2}}\nabla Kg\cdot\nabla Kh d\mathbf x=0,\quad\forall\,g\in\mathcal V,\,h\in\mathcal V^{\perp}.
 \end{equation}
 \end{lemma}
\begin{proof}
Fix $i\in\{1,2\}$. Observe that
\[-\Delta v=\nu_{i}^{-2}v,\quad \forall\,v\in\mathcal V_{i},\]
which implies that
\[Kv=\nu_{i}^{2}v,\quad \forall\,v\in\mathcal V_{i}.\]
Hence for   $g\in\mathcal V_{i}$ and $h\in\mathcal V_{i}^{\perp}$, by integration by parts
\[\int_{\mathbb T^{2}}\nabla Kg\cdot\nabla Kh d\mathbf x=\int_{\mathbb T^{2}} (-\Delta Kh) Kg  d\mathbf x=\int_{\mathbb T^{2}} hKg  d\mathbf x=\nu_{i}^{2}\int_{\mathbb T^{2}} gh  d\mathbf x=0.\]
Hence \eqref{fam1} has been proved. The proof of \eqref{fam2} is similar when $\nu_1=\nu_2=\nu$.
\end{proof}

The following lemma  is mainly used to illustrate Remark \ref{wen1} in Section 1.
\begin{lemma}\label{den1}
Consider a smooth Euler flow on $\mathbb T^{2}$. Let $\tilde{\mathbf v}$, $\tilde \psi $ and $\omega $   be the normalized velocity, the normalized stream function and the vorticity,   respectively. Then   for any $1<p<+\infty,$  it holds that
\begin{equation*}
 \|\tilde\psi\|_{W^{2,p}(D)}\lesssim \|\omega\|_{L^{p}(D)} |\lesssim  \|\tilde\psi\|_{W^{2,p}(D)},
\end{equation*}
\begin{equation*}
 \|\tilde\psi\|_{W^{2,p}(D)}\lesssim \|\tilde{\mathbf v}\|_{W^{1,p}(D)} \lesssim\|\tilde\psi\|_{W^{2,p}(D)}.
\end{equation*}
Here $A\lesssim B$ means $A\leq CB$ for some positive constant $C$ depending only on $\nu_{1}, \nu_{2}$ and $p$.

\end{lemma}
\begin{proof}
Recall that $\tilde\psi$ and $\omega$ satisfy (see Section 1)
   \begin{equation}\label{ppoop}
   \begin{cases}
  -\Delta \tilde\psi=\omega, &\mathbf x\in\mathbb T^2,\\
\int_{\mathbb T^{2} } \tilde\psi d\mathbf x=0.
  \end{cases}
 \end{equation}
 Then the desired estimates are straightforward consequences of the estimate \eqref{bddo}.
\end{proof}

\subsection{Energy-enstrophy inequalities}

In this subsection, we deduce several energy-enstrophy type inequalities for functions in $\mathring{L}^2(\mathbb T^2)$  based on Fourier series expansion.

Recall that   $E$ and $Z$ are defined by   \eqref{deoe} and  \eqref{deoz}, respectively.
It is easy to check that $E $ is well defined in $\mathring L^p(\mathbb T^2)$ for any $1<p<+\infty,$ and $Z$ is well defined in $L^2(\mathbb T^2).$

\begin{lemma}\label{lem666}
Let $f\in \mathring{L}^2(\mathbb T^2)$. Then the following assertions hold.
 \begin{itemize}
 \item[(i)] If $\nu_1<\nu_2$, then
 \[E(f)= \nu_2^{2}  Z(f),\quad \forall\, f\in\mathcal V_2,\]
 \[E(f)\leq \max\left\{ \nu_{1}^{2}, \frac{\nu_{2}^{2}}{4}   \right\}Z(f),\quad \forall\,f\in\mathcal V_2^\perp.\]
  \item[(ii)] If $\nu_1>\nu_2$, then
 \[E(f)= \nu_1^{2}  Z(f),\quad \forall\, f\in\mathcal V_1,\]
 \[E(f)\leq \max\left\{ \nu_{2}^{2}, \frac{\nu_{1}^{2}}{4}   \right\}Z(f),\quad \forall\,f\in\mathcal V_1^\perp.\]
 \item[(iii)] If $\nu_1=\nu_2=\nu$, then
 \[E(f)=\nu^{2}Z(f),\quad \forall\,f\in\mathcal V,\]
 \[E(f)\leq \frac{\nu^{2}}{4}Z(f),\quad\forall\,f\in\mathcal V^\perp.\]
 \end{itemize}

 \end{lemma}
 \begin{proof}
 First we show that for any $f\in \mathring{L}^2(\mathbb T^2)$,
\begin{equation}\label{hae1}
\|\partial_{x_{i}}Kf\|^2_{L^2(\mathbb T^2)}=\sum_{\mathbf k\in\mathbb Z^2,\mathbf k\neq\mathbf 0}\frac{\left( \frac{k_i}{\nu_i}\right)^2|\hat f_{\mathbf k}|^2}{\left[\left( \frac{k_1}{\nu_1}\right)^2+\left( \frac{k_2}{\nu_2}\right)^2\right]^2},\quad i=1,2.
\end{equation}
To prove \eqref{hae1}, we first show that the Fourier series of $\partial_{x_{i}}Kf$ has the form:
\begin{equation}\label{nik1}
\partial_{x_{i}} Kf\thicksim \sum_{\mathbf k\in\mathbb Z^2}c_{\mathbf k}\zeta_{\mathbf k},\quad c_{\mathbf k}=\begin{cases}
0&\mbox{if } \mathbf k=\mathbf 0,\\
\frac{i\left(\frac{k_i}{\nu_i}\right)}{ \left( \frac{k_1}{\nu_1}\right)^2+\left( \frac{k_1}{\nu_1}\right)^2}\hat f_{\mathbf k}&\mbox{if } \mathbf k\neq\mathbf 0.
\end{cases}
\end{equation}
In fact, since the integral of $\partial_{x_{i}}Kf$ on $\mathbb T^2$ is   zero, we have $c_{\mathbf 0}=0$; for $\mathbf k\neq \mathbf 0$, by integration by parts,
\begin{align*}
\bar c_{\mathbf k}=\int_{\mathbb T^2}(\partial_{x_{i}}Kf)  \zeta_{\mathbf k} d\mathbf x&=-\int_{\mathbb T^2} (Kf)  \partial_{x_{i}}\zeta_{\mathbf k}d\mathbf x
\\
&=-i\left(\frac{k_i}{\nu_i}\right)\int_{\mathbb T^2} (Kf)   \zeta_{\mathbf k}d\mathbf x\\
&= -i\left(\frac{k_i}{\nu_i}\right) \left[ \left(\frac{k_1}{\nu_1}\right)^2+\left(\frac{k_2}{\nu_2}\right)^{2}\right]^{-1}\int_{\mathbb T^2}( Kf)(-\Delta\zeta_{\mathbf k})d\mathbf x\\
&=-i\left(\frac{k_i}{\nu_i}\right) \left[ \left(\frac{k_1}{\nu_1}\right)^2+\left(\frac{k_2}{\nu_2}\right)^{2}\right]^{-1}\int_{\mathbb T^2}(-\Delta Kf) \zeta_{\mathbf k}d\mathbf x\\
&=-i\left(\frac{k_i}{\nu_i}\right) \left[ \left(\frac{k_1}{\nu_1}\right)^2+\left(\frac{k_2}{\nu_2}\right)^{2}\right]^{-1}\int_{\mathbb T^2}f \zeta_{\mathbf k} d\mathbf x\\
&=-i\left(\frac{k_i}{\nu_i}\right) \left[ \left(\frac{k_1}{\nu_1}\right)^2+\left(\frac{k_2}{\nu_2}\right)^{2}\right]^{-1}\hat f_{\mathbf k}.
\end{align*}
Here we used \eqref{qiud} and \eqref{qiud8}. Hence \eqref{nik1} has been proved.
From \eqref{nik1}, we can apply Parseval's identity (see \cite{Gra}, Proposition 3.2.7) to obtain
\[\|\partial_{x_{i}}Kf\|^2_{L^2(\mathbb T^2)}=\sum_{\mathbf k\in\mathbb Z^2} |c_{\mathbf k}|^2=\sum_{\mathbf k\in\mathbb Z^2,\mathbf k\neq\mathbf 0}\frac{\left(\frac{k_i}{\nu_i}\right)^2|\hat f_{\mathbf k}|^2}{\left[\left(\frac{k_1}{\nu_1}\right)^2+\left(\frac{k_2}{\nu_2}\right)^2\right]^2},\]
which is exactly \eqref{hae1}.

As a consequence of \eqref{hae1}, we obtain
 \begin{equation}\label{hae2}
 E(f)= \frac{1}{2}\|\nabla Kf\|_{L^2(\mathbb T^2)}^2=\frac{1}{2}\sum_{\mathbf k\in\mathbb Z^2,\mathbf k\neq\mathbf 0}\frac{ |\hat f_{\mathbf k}|^2}{ \left(\frac{k_1}{\nu_1}\right)^2+\left(\frac{k_2}{\nu_2}\right)^2 },\quad \forall\,f\in \mathring{L}^2(\mathbb T^2).
 \end{equation}
 Besides, by Parseval's  identity, the enstrophy can also be expressed in terms of Fourier coefficients:
 \begin{equation}\label{hae3}
 Z(f)=\frac{1}{2}\|f\|_{L^2(\mathbb T^2)}^2=\frac{1}{2}\sum_{\mathbf k\in\mathbb Z^2,\mathbf k\neq\mathbf 0} |\hat f_{\mathbf k}|^2,\quad \forall\,f\in  \mathring{L}^2(\mathbb T^2).
 \end{equation}

Below we prove the lemma based on \eqref{hae2} and \eqref{hae3}.
We only prove (i), since the proofs of (ii) and (iii)  are almost identical to that of (i). Notice that $f\in\mathcal V_2$ if and only if $\hat f_{\mathbf k}=0$ for $\mathbf k\neq ( 0,\pm1)$, and  $f\in\mathcal V^\perp_2$ if and only if $\hat f_{\mathbf k}=0$ for $\mathbf k=(0,\pm1)$.
Hence for   $f\in \mathcal V_2,$
 \[E(f) =\frac{1}{2}\sum_{\mathbf k=( 0,\pm1)} \frac{ |\hat f_{\mathbf k}|^2}{ \left(\frac{k_1}{\nu_1}\right)^2+ \left(\frac{k_2}{\nu_2}\right)^2 }=\frac{1}{2}\nu_{2}^{2}\sum_{\mathbf k=( 0,\pm1)}   |\hat f_{\mathbf k}|^2 =\nu_{2}^{2}Z(f).\]
For  $f\in \mathcal V_2^\perp,$
\begin{equation}\label{mlw0}
E(f) =\frac{1}{2}\sum_{\mathbf k\neq(0,\pm1),\mathbf k\neq \mathbf 0} \frac{ |\hat f_{\mathbf k}|^2}{  \left(\frac{k_1}{\nu_1}\right)^2+ \left(\frac{k_2}{\nu_2}\right)^2 }.
\end{equation}
We claim that
\begin{equation}\label{mlw}
\left(\frac{k_1}{\nu_1}\right)^2+ \left(\frac{k_2}{\nu_2}\right)^2\geq \min\left\{\frac{1}{\nu_1^2}, \frac{4}{\nu_2^2}\right\},\quad \forall\, \mathbf k\neq (0,\pm1), \mathbf k\neq\mathbf 0.
\end{equation}
In fact, notice that  $\mathbf k\neq (0,\pm1),$ $ \mathbf k\neq\mathbf 0$ if and only if
\[|k_1|\geq 1 \quad \mbox{or}\quad |k_2|\geq 2.\]
If $|k_1|\geq 1$,  then
\[\left(\frac{k_1}{\nu_1}\right)^2+ \left(\frac{k_2}{\nu_2}\right)^2\geq \frac{1}{\nu_1^2};\]
if $|k_2|\geq 2$, then
\[\left(\frac{k_1}{\nu_1}\right)^2+ \left(\frac{k_2}{\nu_2}\right)^2\geq \frac{4}{\nu_2^2}.\]
Hence \eqref{mlw} follows.
Combining \eqref{mlw0} and \eqref{mlw}, we have  that
  \begin{align*}
  E(f) &\leq \frac{1}{2} \sum_{\mathbf k\neq(0,\pm1),\mathbf k\neq \mathbf 0} \frac{ |\hat f_{\mathbf k}|^2}{ \min\left\{\frac{1}{\nu_1^2}, \frac{4}{\nu_2^2}\right\}}\\
  &= \frac{1}{2}\max\left\{\nu_{1}^{2},\frac{\nu_{2}^{2}}{4}  \right\} \sum_{\mathbf k\neq(0,\pm1),\mathbf k\neq \mathbf 0}   |\hat f_{\mathbf k}|^2 \\
  &=   \max\left\{ \nu_{1}^{2},\frac{\nu_{2}^{2}}{4}    \right\}Z(f).
  \end{align*}
 \end{proof}

 \section{Variational characterizations}

 Throughout this section, let  $A, B\geq 0$ be fixed.  Denote
\begin{equation}\label{jk1}
v_{A,i}=A\sin\left(\frac{x_i}{\nu_i}\right),\quad i=1,2,
\end{equation}
\begin{equation}\label{jk2}
v_{A,B}=A\sin\left(\frac{x_1}{\nu}\right)+B\sin\left(\frac{x_2}{\nu}\right)\quad\mbox{if }\nu_1=\nu_2=\nu.
\end{equation}
 Denote by $\mathcal R_{A,i}, \mathcal R_{A,B}$ the set of rearrangements of $v_{A,i}, v_{A,B}$ on $\mathbf T^2$, respectively, i.e.,
 \begin{equation}\label{jk3}
 \mathcal R_{A,i}=\mathcal R(v_{A,i}),\quad i=1,2,
 \end{equation}
\begin{equation}\label{jk4}
\mathcal R_{A,B}=\mathcal R(v_{A,B})\quad\mbox{if }\nu_1=\nu_2=\nu.
\end{equation}
 It is easy to check that  $\mathcal V_{A,i}\subset \mathcal R_{A,i}$, $i=1,2,$ and   $\mathcal V_{A,B}\subset\mathcal R_{A,B}$ if $\nu_1=\nu_2=\nu.$

Consider the following variational problems:
\begin{equation}\label{jk5}
\mathsf m_{A,i}=\sup_{v\in\mathcal R_{A,i}}E(v),\quad i=1,2,
\end{equation}
\begin{equation}\label{jk6}
\mathsf m_{A,B}=\sup_{v\in\mathcal R_{A,B}}E(v)\quad \mbox{if } \nu_1=\nu_2=\nu.
\end{equation}

 Our aim in this section is to prove the following variational characterizations for $\mathcal V_{A,1}$, $\mathcal V_{A,2}$ and $\mathcal R_{A,B}$.

\begin{proposition}\label{varcha}
Let $\mathcal R_{A,1},\mathcal R_{A,2},\mathcal R_{A,B}$ be defined by \eqref{jk3}, \eqref{jk4}, and $\mathsf m_{A,1},\mathsf m_{A,2}, \mathsf m_{A,B}$ be defined by \eqref{jk5}, \eqref{jk6}.
\begin{itemize}
\item[(i)] If $\nu_1<\nu_2$, then
\[\mathcal V_{A,2}=\{v\in\mathcal R_{A,2}\mid E(v)=\mathsf m_{A,2}\}.\]
\item [(ii)]If $\nu_1>\nu_2$, then
\[\mathcal V_{A,1}=\{v\in\mathcal R_{A,1}\mid E(v)=\mathsf m_{A,1}\}.\]
\item[(iii)] If $\nu_1=\nu_2=\nu$, then
\[\mathcal V_{A,B}=\{v\in\mathcal R_{A,B}\mid E(v)=\mathsf m_{A,B}\}.\]
\end{itemize}
\end{proposition}

\begin{proof}

First we prove (i).   Let $v_{A,2}$ be given by \eqref{jk1}.  Denote
\[ Z_A=Z(v_{A,2}).\]
Then $ Z(v )=Z_A$ for any $v\in \mathcal R_{A,2}.$
Decompose any $v\in\mathcal R_{A,2}$ into two components:
\begin{equation}\label{cbnj}
v= \bar v+\tilde v, \quad \bar v\in \mathcal V_2,\,\,\tilde v\in\mathcal V_2^\perp.
\end{equation}
It is clear that
\begin{equation}\label{lb1}
Z(\bar v)+Z(\tilde v)=Z_A.
\end{equation}
Using Lemma \ref{fam}, we have that
\begin{equation}\label{lb2}
\begin{split}
E(v)&=\frac{1}{2}\int_{L^2(\mathbb T^2)} |\nabla Kv |^2d\mathbf x\\
&=\frac{1}{2}\int_{L^2(\mathbb T^2)} |\nabla K\bar v |^2d\mathbf x+ \int_{L^2(\mathbb T^2)}  \nabla K\bar v\cdot\nabla K\tilde v d\mathbf x+\frac{1}{2}\int_{L^2(\mathbb T^2)} |\nabla K\tilde v |^2d\mathbf x\\
&=E(\bar v)+E(\tilde v).
\end{split}
\end{equation}
Recalling Lemma \ref{lem666}(i), it holds that
\begin{equation}\label{lb3}
E(\bar v)=\nu_2^2Z(\bar v),\quad E(\tilde v)\leq  \max\left\{ \nu_{1}^{2},\frac{\nu_{2}^{2}}{4}    \right\}Z(\tilde v).
\end{equation}
Combining \eqref{lb1}-\eqref{lb3}, we obtain
\begin{equation}\label{lb4}
E( v)\leq \nu_2^2Z_A,
\end{equation}
and the  equality holds if and only if $v\in\mathcal V_2$. In other words, we have proved that
\[\mathsf m_{A,2}=\nu_2^2Z_A,\]
and, moreover, for any $v\in\mathcal R_{A,2}$,    $E(v)=\mathsf m_{A,2}$ if and only if $v\in\mathcal V_2$.  To finish the proof of (i), it is sufficient to show that
\[\mathcal V_2\cap \mathcal R_{A,2}=\mathcal V_{A,2}.\]
The inclusion $\mathcal V_{A,2}\subset \mathcal V_2\cap \mathcal R_{A,2}$ is obvious. To prove the inverse inclusion, it is sufficient to show that for any $v\in \mathcal V_2\cap \mathcal R_{A,2}$ with the form
\[v=B\sin\left(\frac{x_2}{\nu_2}+\beta\right)\]
for some $B\geq 0$ and $\beta\in\mathbb R,$
it holds that $B=A.$ This is obvious since
\[B=\|v\|_{L^\infty(\mathbb T^2)}=\|v_A\|_{L^\infty(\mathbb T^2)}=A.\]

The proof of (ii) is almost identical to that of (i), therefore we omit it.

Now we prove (iii).  Denote
\[Z_{A,B}=Z(v_{A,B}).\]
Then it is clear that $ Z(v)=Z_{A,B}$ for any $v\in \mathcal R_{A,B}.$
Analogously to \eqref{cbnj}, we decompose any $v\in\mathcal R_{A,B}$ into two components:
\[v= \bar v+\tilde v, \quad \bar v\in \mathcal V,\,\,\tilde v\in\mathcal V^\perp.\]
Then
\begin{equation}\label{igl1}
Z(\bar v)+Z(\tilde v)=Z_{A,B}.
\end{equation}
As in \eqref{lb2}, we can prove that
\begin{equation}\label{igl2}
E(v)=E(\bar v)+E(\tilde v).
\end{equation}
Moreover, by Lemma \ref{lem666}(iii),
\begin{equation}\label{igl3}
 E(\bar v)=\nu^2Z(\bar f),\quad E(\tilde v)\leq \frac{\nu^2}{4}Z(\tilde v).
\end{equation}
Therefore we infer from \eqref{igl1}-\eqref{igl3} that
\begin{equation}\label{igl4}
 E(v)\leq   {\nu^2} Z_{A,B},
\end{equation}
and the equality holds if and only if $v\in \mathcal V.$
Hence we have proved that
\[\mathsf m_{A,B}=\nu^2Z_{A,B},\]
and, moreover, for any $v\in\mathcal R_{A,B},$ $v$ is a maximizer of $E$ relative to $\mathcal R_{A,B}$ if and only if $v\in \mathcal V$. To finish the proof, it is sufficient to show that
\[\mathcal V\cap \mathcal R_{A,B}=\mathcal V_{A,B}\cup\mathcal V_{B,A}.\]
Since it is obvious that   $\mathcal V_{A,B}, \mathcal V_{B,A}\subset   \mathcal R_{A,B}$, it holds that
 $\mathcal V_{A,B}\cup\mathcal V_{B,A}\subset \mathcal V\cap \mathcal R_{A,B}$. To prove the inverse inclusion, it is sufficient to show that for any $v\in \mathcal V\cap \mathcal R_{A,B}$ with the form
\[v=C\sin\left(\frac{x_1}{\nu}+\alpha\right)+D\sin\left(\frac{x_2}{\nu}+\beta\right)\]
for some $C,D\geq0$ and $\alpha,\beta\in \mathbb R$, it holds that
 \begin{equation}\label{twis0}
A=C,\,B=D\quad\mbox{or}\quad A=D, \,B=C.
\end{equation}
To prove \eqref{twis0}, notice that for any $v\in\mathcal R_{A,B},$
\begin{equation}\label{rav1}
\|v\|_{L^\infty(\mathbb T^2)}=\|v_{A,B}\|_{L^\infty(\mathbb T^2)},\quad
\|v\|_{L^2(\mathbb T^2)}=\|v_{A,B}\|_{L^2(\mathbb T^2)},
\end{equation}
which implies that
 \begin{equation}\label{twis1}
 A+B=C+D,\quad
 A^2+B^2=C^2+D^2.
 \end{equation}
From \eqref{twis1}, we can easily obtain \eqref{twis0}.  In fact, \eqref{twis1} can be written as
 \begin{equation}\label{hxi1}
 A-C=D-B,\quad
(A-C)(A+C)=(D-B)(D+B).
 \end{equation}
 If $A-C =0$, then $A=C,$ $B=D$, hence \eqref{twis0}  holds; if $A-C\neq0$, then $A+C=D+B$, which together with $A+B=C+D$ gives $A=D, B=C$, hence \eqref{twis0} still holds.
\end{proof}

\begin{remark}
From the above proof, for $i=1,2,$ $\mathcal V_{A,i}$ is in fact the set of maximizers of $E$ relative to
\[\{v\in \mathring L^{p}(\mathbb T^{2})\mid Z(v)=Z_A\}.\]
However, $\mathcal V_{A,B}$ does not have such a characterization. To distinguish $\mathcal V_{A,B}$ from other sinusoidal flows, it is necessary to study their rearrangements.
\end{remark}

\section{Compactness}

Throughout this paper, let $A,B\geq0,$ $1<p<+\infty$ be fixed.

Our purpose in this section is to prove the following proposition, stating that any maximizing sequence for  the maximization problem  \eqref{jk5} or \eqref{jk6} is compact in $L^p(\mathbb T^2)$.

\begin{proposition}\label{compact}
Let $\mathcal R_{A,1},\mathcal R_{A,2},\mathcal R_{A,B}$ be defined by \eqref{jk3}, \eqref{jk4}, and $\mathsf m_{A,1},\mathsf m_{A,2}, \mathsf m_{A,B}$ be defined by \eqref{jk5}, \eqref{jk6}.
\begin{itemize}
\item[(i)] If $\nu_1<\nu_2$, then for any sequence $\{v_n\} \subset\mathcal R_{A,2}$ satisfying
\begin{equation}\label{bnr1}
\lim_{n\to+\infty}E(v_n)=\mathsf m_{A,2},
\end{equation}
there exists some subsequence of $\{v_n\},$ denoted by $\{v_{n_j}\}$, and some $\hat v\in\mathcal V_{A,2},$ such that $v_{n_j}\to \hat v$ in $L^p(\mathbb T^2)$ as $j\to+\infty$.

\item [(ii)] If $\nu_1>\nu_2$, then for any sequence $\{v_n\} \subset\mathcal R_{A,1}$ satisfying
\begin{equation}\label{bnr2}
\lim_{n\to+\infty}E(v_n)=\mathsf m_{A,1},
\end{equation}
there exists some subsequence of $\{v_n\},$ denoted by $\{v_{n_j}\}$, and some $\hat v\in\mathcal V_{A,1},$ such that $v_{n_j}\to \hat v$ in $L^p(\mathbb T^2)$ as $j\to+\infty$.
\item[(iii)] If $\nu_1=\nu_2=\nu$, then for any sequence $\{v_n\} \subset\mathcal R_{A,B}$ satisfying
\begin{equation}\label{bnr3}
\lim_{n\to+\infty}E(v_n)=\mathsf m_{A,B},
\end{equation}
there exists some subsequence of $\{v_n\},$ denoted by $\{v_{n_j}\}$, and some $\hat v\in\mathcal V_{A,B},$ such that $v_{n_j}\to \hat v$ in $L^p(\mathbb T^2)$ as $j\to+\infty$.
\end{itemize}

\end{proposition}

To prove Proposition \ref{compact},   we need several  lemmas.
\begin{lemma}\label{ym0}
 For any $1<p<+\infty,$ $E$ is sequentially weakly continuous in $\mathring L^{p}(\mathbb T^{2})$, i.e., if $v_n$ converges weakly to $\hat v$ in  $\mathring L^p(\mathbb T^2)$, then \[\lim_{n\to+\infty}E(v_n)=E(\hat v).\]
\end{lemma}
\begin{proof}
Since $K$ is bounded from $\mathring L^p(\mathbb T^2)$ to $\mathring W^{2,p}(\mathbb T^2)$,  $Kv_n$ converges weakly to $K\hat v$ in $ W^{2,p}(\mathbb T^2)$. Taking into account the fact that the embedding $ W^{2,p}(\mathbb T^2)\hookrightarrow L^\infty(\mathbb T^2)$ is compact, we further deduce that $Kv_n$ converges strongly to $K\hat v$ in $ L^\infty (\mathbb T^2)$. Hence the desired result follows immediately.
\end{proof}

\begin{lemma}[\cite{B1}, Theorem 6]\label{lem201}
Let  $\mathcal R(f_0) $ be set of  rearrangements  of some $f_0\in L^{p}(\mathbb T^{2})$  on $\mathbb T^{2}$, and $\overline {\mathcal R(f_0)}$ be the weak closure of $\mathcal R(f_0) $ in $L^p(\mathbb T^{2}).$ Then $\overline {\mathcal R(f_0)}$ is convex, i.e., $\theta f_1+(1-\theta)f_2\in \overline {\mathcal R(f_0)}$ whenever $f_1, f_2\in \overline {\mathcal R(f_0)}$ and $\theta\in[0,1].$
\end{lemma}

\begin{lemma}[\cite{B1}, Theorem 4]\label{lem202}
 Let $p^*=p/(p-1)$ be the H\"older conjugate of $p$.  Let $\mathcal R({f_0}),$ $\mathcal R({g_0})$ be sets of  rearrangements on $\mathbb T^{2}$ of some $f_{0}\in L^p(\mathbb T^{2})$ and some $g_{0}\in L^{q}(\mathbb T^{2})$, respectively. Then for any $\tilde g\in\mathcal R({g_0}),$ there exists $\tilde v\in \mathcal R({f_0})$, such that
\[\int_{\mathbb T^{2}} \tilde f \tilde gdx\geq \int_{\mathbb T^{2}}fgdx,\quad\forall\, f\in {\mathcal R}({f_0}),\,\,g\in {\mathcal R}({g_0}).\]
\end{lemma}

Now we are ready to prove Proposition \ref{compact}.
\begin{proof}[Proof of Proposition \ref{compact}]

To prove (i), fix a sequence $\{v_n\}\subset \mathcal R_{A,2}$ such that \eqref{bnr1} holds. Obviously $\{v_n\}$ is bounded in $L^p(\mathbb T^2).$ Without loss of generality, we can assume, up to a subsequence, that $v_n$ converges weakly to some $\hat v\in \overline{\mathcal R_{A,2}}$ in $L^p(\mathbb T^2).$ Here $\overline{\mathcal R_{A,2}}$ is the weak closure of $ {\mathcal R}_{A,2}$ in $L^p(\mathbb T^2)$ as in Lemma \ref{lem201}.   By Lemma \ref{ym0},  we have that
\begin{equation}\label{reny0}
E(\hat v)=\mathsf m_{A,2}\geq E(v),\quad\forall\,v\in \overline{\mathcal R_{A,2}}.\end{equation}
Here we used the  fact that
\[\mathsf m_{A,2}=\sup_{v\in\mathcal R_{A,2}}E(v)=\sup_{v\in\overline{\mathcal R_{A,2}}}E(v).\]
By Lemma \ref{lem201}, $\overline{\mathcal R_{A,2}}$ is convex. Hence for any $v\in \mathcal R_{A,2}$ and $\theta\in[0,1],$
we have that $\theta v+(1-\theta)\hat v\in \overline{\mathcal R_{A,2}}.$ By \eqref{reny0}, $E(\theta v+(1-\theta)\hat v)$ attains its maximum value at $\theta=0$. Therefore we have that
\[\frac{d}{d\theta}E(\theta v+(1-\theta)\hat v)\bigg|_{\theta=0^+}\leq 0,\]
which gives
\begin{equation}\label{reny}
\int_{\mathbb T^2}vK\hat vd\mathbf x\leq \int_{\mathbb T^2}\hat vK\hat vd\mathbf x.
\end{equation}
Note that \eqref{reny} holds for any $v\in\mathcal R_{A,2}.$
By Lemma \ref{lem202}, there exists some $\tilde v\in\mathcal R_{A,2}$ such that
\begin{equation}\label{reny7}
\int_{\mathbb T^2} vK\hat vd\mathbf x\leq \int_{\mathbb T^2}\tilde vK\hat vd\mathbf x,\quad\forall\,v\in\mathcal R_{A,2}.
\end{equation}
By a simple approximation procedure, it is easy to show that \eqref{reny7} actually holds for any $v\in\overline{\mathcal R_{A,2}}.$
In particular,
\begin{equation}\label{reny9}
\int_{\mathbb T^2} \hat vK\hat vd\mathbf x\leq \int_{\mathbb T^2}\tilde vK\hat vd\mathbf x.
\end{equation}
Below we show that $\tilde v=\hat v.$
We compute as follows:
\begin{equation}\label{compp}
\begin{split}
 E(\tilde v-\hat v)&=E(\tilde v)+E(\hat v)-\int_{\mathbb T^2} \tilde vK\hat vd\mathbf x\\
&\leq E(\tilde v)+E(\hat v)-\int_{\mathbb T^2} \hat vK\hat vd\mathbf x\\
&=\frac{1}{2}E(\tilde v)-\frac{1}{2}E(\hat v)\\
&\leq 0.
\end{split}
\end{equation}
Here we used \eqref{reny0}, \eqref{reny9}, and  the fact that $K$ is symmetric (see Lemma \ref{sppp}).
Taking into account the fact that $K$ is positive definite (see Lemma \ref{sppp}), we infer from \eqref{compp} that  $\tilde v=\hat v.$  In particular, $\hat v \in\mathcal R_{A,2}.$ Furthermore, since $E(\hat v)=\mathsf m_{A,2}$ (recall \eqref{reny0}), we can apply Proposition \ref{varcha}(i) to obtain
\[\hat v\in\mathcal V_{A,2}.\]

To conclude, we have proved that  $v_n$, up to a subsequence, converges weakly to some $\hat v\in\mathcal V_{A,2}$ in $L^p(\mathbb T^2)$ as $n\to+\infty.$ In particular, $\|v_n\|_{L^p(\mathbb T^2)}=\|\hat v\|_{L^p(\mathbb T^2)}$ for all $n$.
By uniform convexity, we further deduce that $v_n$, up to a subsequence, in fact converges \emph{strongly} to  $\hat v$ in $L^p(\mathbb T^2)$ as $n\to+\infty.$ This completes the proof of (i).

The proofs of (ii)(iii) are almost identical to that of (i), we omit them therefore.

\end{proof}

\section{Proof of  Theorem \ref{thm1}}

With the variational characterizations for $\mathcal V_{A,1},$  $\mathcal V_{A,2}$ and  $\mathcal V_{A,B}$ established in Section 3, and the compactness proved in Section 4, we are ready to prove Theorem \ref{thm1} in this section.

Throughout this section, let $A,B\geq 0$ and $1<p<+\infty$ be fixed.

\begin{proof}[Proof of Theorem \ref{thm1}(i)]
Suppose by contradiction that $\mathcal V_{A,2}$ is not nonlinearly stable in the sense of \eqref{sfg1}. Then there exist  some $\varepsilon_0>0$, a sequence of smooth Euler flows on $\mathbb T^2$ with  vorticity $\{\omega^n\}$, and a sequence of times $\{t_n\}$, such that
\begin{equation}\label{cc1}
\lim_{n\to+\infty}\min_{v\in\mathcal V_{A,2}}\|\omega^n_0 -v\|_{L^p(\mathbb T^2)}=0,
\end{equation}
\begin{equation}\label{cc2}
\min_{v\in\mathcal V_{A,2}}\|\omega^n_{t_n}-v\|_{L^p(\mathbb T^2)}\geq \varepsilon_0,\quad\forall\,n.
\end{equation}
Here $\omega^n_t:=\omega^n(t,\cdot).$

It is easy to check that $\mathcal V_{A,2}$ is compact in $L^p(\mathbb T^2)$ (which can also be seen from Proposition \ref{compact}(i)), hence there exist  some subsequence of $\{\omega^n_0\}$, still denoted by $\{\omega^n_0\}$, and some $\bar\omega\in\mathcal V_{A,2}$, such that
\begin{equation}\label{foop0}
\lim_{n\to+\infty}\|\omega^n_0-\bar\omega\|_{L^p(\mathbb T^2)}=0.
\end{equation}
Consequently,
\begin{equation}\label{enc1}
\lim_{n\to+\infty}E(\omega^n_0)=E(\bar\omega)=\mathsf m_{A,2}.
\end{equation}
By  energy conservation, we get from \eqref{enc1} that
\begin{equation}\label{enc2}
\lim_{n\to+\infty}E(\omega^n_{t_n})=\lim_{n\to+\infty}E(\omega^n_0)= \mathsf m_{A,2}.
\end{equation}

Now we can easily get a contradiction if we only consider perturbed flows with vorticity on $\mathcal R_{A,2}$.
In fact, if   $\{\omega_{t_n}^n\}\subset \mathcal R_{A,2}$ for any $n$,  then we can choose $v_n=\omega_{t_n}^n$ in  Proposition \ref{compact}(i) (note that \eqref{bnr1} is satisfied by \eqref{enc2}) to deduce that
$\{\omega^n_{t_n}\}$, up to a subsequence, converges to some element in $\mathcal V_{A,2}$ in $L^p(\mathbb T^2)$ as $n\to+\infty$. This obviously contradicts   \eqref{cc2}.

To deal with the general case, we  need to introduce a sequence of ``followers" to  $\{\omega^n\}$ as in \cite{B5, BR}. For fixed $n$,  denote by $\mathbf v^n$ the velocity of the Euler flow with vorticity $\omega^n$. Then $\omega^{n}$ satisfies the following nonlinear transport equation  (see \cite{MB}, p. 20):
\begin{equation}\label{tspe2}
\begin{cases}
\partial_{t}\omega^{n}+\mathbf v^{n}\cdot\nabla\omega^{n}=0,&t>0,\,\mathbf x\in\mathbb T^{2},\\
 \omega^{n}(0,\cdot)=\omega^{n}_{0}.
\end{cases}
\end{equation}
Let $\zeta^{n}$ be the solution of the following  linear transport equation:
\begin{equation}\label{tspe}
\begin{cases}
\partial_{t}\zeta^{n}+\mathbf v^{n}\cdot\nabla\zeta^{n}=0,&t>0,\,\mathbf x\in\mathbb T^{2},\\
 \zeta^{n}(0,\cdot)=\bar\omega.
\end{cases}
\end{equation}
For simplicity, denote $\zeta^{n}_{t}=\zeta^{n}(t,\cdot)$. Since $\mathbf v^{n}$ is divergence-free, 
by the Liouville theorem (see \cite{MPu}, p. 48), it holds that 
\begin{equation}\label{foop1}
\zeta^n_t\in\mathcal R(\bar\omega)=\mathcal R_{A,2},\quad\forall\,t\geq 0,
\end{equation}
On the other hand, combining  \eqref{tspe2} and \eqref{tspe}, we see that $\zeta^{n}-\omega^{n}$ satisfies
 \begin{equation}\label{tspe3}
\begin{cases}
\partial_{t}(\zeta^{n}-\omega^{n})+\mathbf v^{n}\cdot\nabla(\zeta^{n}-\omega^{n})=0,&t>0,\,\mathbf x\in\mathbb T^{2},\\
(\zeta^{n}-\omega^{n})(0,\cdot)=\bar\omega-\omega^{n}_{0}.
\end{cases}
\end{equation}
Again, by the Liouville theorem,
\begin{equation}\label{foop2}
\zeta^n_t-\omega^n_t\in\mathcal R( \bar\omega-\omega^n_0),\quad\forall\,t\geq 0.
\end{equation}

Having introduced the sequence of  ``followers" $\{\zeta^n\}$, we are ready to deduce a contradiction.
By \eqref{foop0} and \eqref{foop2},
\begin{equation}\label{foop5}
\lim_{n\to+\infty}\|\zeta^n_{t_n}-\omega^n_{t_n}\|_{L^p(\mathbb T^2)}=0,
\end{equation}
which together with \eqref{enc2} implies that
\begin{equation}\label{foop6}
\lim_{n\to+\infty}E(\zeta^n_{t_n})= \mathsf m_{A,2}.
\end{equation}
To conclude, we have found a sequence $\{\zeta^n_{t_n}\}$ such that \eqref{foop1} and \eqref{foop6} hold. Applying Proposition \ref{compact}(i), we infer that
$\zeta^n_{t_n}$, up to a subsequence, converges to some $\tilde \omega\in \mathcal V_{A,2}$ in $L^p(\mathbb T^2)$ as $n\to+\infty$. Combining \eqref{foop5}, we deduce that $\omega^n_{t_n}$ converges to $\tilde \omega$ in $L^p(\mathbb T^2)$ as $n\to+\infty$, which obviously contradicts \eqref{cc2}.

\end{proof}

The proof of Theorem \ref{thm1}(ii) can also be proved similarly.

To prove Theorem \ref{thm1}(iii), we need the following lemma.
\begin{lemma}\label{pddd9}
If $\nu_1=\nu_2=\nu$ and $A\neq B$, then
\[\min_{u\in\mathcal V_{A,B},v\in\mathcal V_{B,A}}\|u-v\|_{L^p(\mathbb T^2)}>0.\]
\end{lemma}
\begin{proof}
Observe that
\[\mathcal V_{A,B}=\left\{a\sin\left(\frac{x_1}{\nu}\right)+b\cos\left(\frac{x_1}{\nu}\right)+c\sin\left(\frac{x_2}{\nu} \right)+d\cos\left(\frac{x_2}{\nu} \right) \mid a^2+b^2=A^2,\,c^2+d^2=B^2 \right\}.\]
Hence it is easy to see that $\mathcal V_{A,B}$ is compact in $L^p(\mathbb T^2).$ Similarly,  $\mathcal V_{B,A}$ is also  compact in $L^p(\mathbb T^2).$
To finish the proof, it is sufficient to show that $\mathcal V_{A,B}\cap  \mathcal V_{B,A}=\varnothing.$

Fix $u\in \mathcal V_{A,B}$, $v\in\mathcal V_{B,A}$.  Assume that $u,v$ has the form
\[u=A\sin\left(\frac{x_1}{\nu}+\alpha\right)+B\sin\left(\frac{x_2}{\nu}+\beta\right),\quad v=B\sin\left(\frac{x_1}{\nu}+\alpha'\right)+A\sin\left(\frac{x_2}{\nu}+\beta'\right),\]
where $\alpha,\beta,\alpha',\beta'\in\mathbb R.$ Then $u,v$ can be written as
\begin{align*}
u=A\cos\alpha\sin\left(\frac{x_1}{\nu}\right)+A\sin\alpha\cos\left(\frac{x_1}{\nu}\right)+B\cos\beta\sin\left(\frac{x_2}{\nu}\right)+B\sin\beta\cos\left(\frac{x_2}{\nu}\right),
\end{align*}
\begin{align*}
v=B\cos\alpha'\sin\left(\frac{x_1}{\nu}\right)+B\sin\alpha'\cos\left(\frac{x_1}{\nu}\right)+A\cos\beta'\sin\left(\frac{x_2}{\nu}\right)+A\sin\beta'\cos\left(\frac{x_2}{\nu}\right).
\end{align*}
If $u=v$, then we must have
\[A\cos\alpha=B\cos\alpha', \quad A\sin\alpha=B\sin\alpha',\quad B\cos\beta=A\cos\beta',\quad B\sin\beta=A\sin\beta',\]
which implies that  $A=B,$ a contradiction.
\end{proof}

 \begin{proof}[Proof of Theorem \ref{thm1}(iii)]

Following the proof of Theorem \ref{thm1}(i), we can show that $\mathcal V_{A,B}\cup\mathcal V_{B,A}$ is nonlinearly stable in the sense of \eqref{sfg1}.  If $A=B$, then $\mathcal V_{A,B}=\mathcal V_{A,B}\cup\mathcal V_{B,A}$  is nonlinearly stable in the sense of \eqref{sfg1}. If $A=B$, then by Lemma \ref{pddd9}  there is a positive distance between $\mathcal V_{A,B}$ and $\mathcal V_{B,A}$ in $L^p(\mathbb T^2)$, hence by continuity each of them is nonlinearly stable in the sense of \eqref{sfg1}.

 \end{proof}

 \appendix

 \section{Proof of Theorem \ref{thm0}}\label{appe1}

 In this appendix, we give the proof of Theorem \ref{thm0} based on the  energy-enstrophy  type  inequalities  established in Lemma \ref{lem666}.
 We only prove Theorem \ref{thm0}(i). The other two items can be proved in a similar manner.

For any Euler flow with vorticity $\omega$ and velocity mean vector $\mathbf b$,  since $E$ and $Z$ are both conserved quantities, we have that
\begin{equation}\label{gkl000}
\nu_2^2Z(\omega_t)- E(\omega_t)=\nu_2^2Z(\omega_t)- E(\omega_0),\quad\forall\,t\geq 0.
\end{equation}
Decompose $\omega_t$ into two components:
\[\omega_t=\bar \omega_t+\tilde \omega_t,\quad\bar\omega_t\in \mathcal V_2,\,\tilde \omega_t\in\mathcal V_2^\perp.\]
Using Lemma \ref{fam}, we infer from \eqref{gkl000} that
\begin{equation}\label{gkl1}
 \nu_2^2Z(\bar \omega_t)+\nu_2^2Z(\tilde \omega_t)-E(\bar \omega_t) -E(\tilde \omega_t)=\nu_2^2Z(\bar \omega_0)+\nu_2^2Z(\tilde \omega_0)-E(\bar \omega_0)-E(\tilde \omega_0),\quad\forall\,t\geq0.
\end{equation}
Applying Lemma \ref{lem666}(i), we have that
\begin{equation}\label{gkl2}
E(\bar\omega_t)= \nu_2^{2}  Z(\bar\omega_t),\quad \forall\,t\geq 0.
\end{equation}
\begin{equation}\label{gkl3}
E(\tilde \omega_t)\leq \max\left\{ \nu_{1}^{2}, \frac{\nu_{2}^{2}}{4}   \right\}Z(\tilde \omega_t),\quad
\forall\,t\geq 0.
\end{equation}
From \eqref{gkl1} and   \eqref{gkl2},   we have that
\begin{equation}\label{gkl4}
  \nu_2^2Z(\tilde \omega_t)-E(\tilde \omega_t)= \nu_2^2Z(\tilde \omega_0) -E(\tilde \omega_0),\quad\forall\,t\geq0.
\end{equation}
which together with \eqref{gkl3} gives
\begin{equation}\label{gkl5}
\begin{split}
  \left(\nu_2^2-\max\left\{ \nu_{1}^{2}, \frac{\nu_{2}^{2}}{4}   \right\}\right)Z(\tilde \omega_t)&\leq \nu_2^2Z(\tilde \omega_0) -E(\tilde \omega_0)\\
  &\leq \nu_2^2Z(\tilde \omega_0),\quad\forall\,t\geq0,
  \end{split}
\end{equation}
or equivalently,
\begin{equation}\label{gkl50}
Z(\tilde \omega_t)\leq C_{\nu_1,\nu_2}Z(\tilde \omega_0)\quad\forall\,t\geq0,\quad C_{\nu_1,\nu_2}:=  \left(1-\max\left\{\left(\frac{\nu_1}{\nu_2}\right)^2, \frac{1}{4}   \right\}\right)^{-1}>0.
\end{equation}
Taking into account the fact that
\[\min_{v\in\mathcal V_2}\|\omega_t-v\|_{L^2(\mathbb T^2)}=\|\tilde \omega_t\|_{L^2(\mathbb T^2)}=\sqrt{2Z(\tilde\omega_t)},\quad\forall\,t\geq 0,\]
we obtain the desired stability from \eqref{gkl50}  immediately.

\bigskip

 {\bf Acknowledgements:}
{G. Wang was supported by National Natural Science Foundation of China (12001135, 12071098) and China Postdoctoral Science Foundation (2019M661261, 2021T140163). B. Zuo was supported by National Natural Science Foundation of China (12101154).}

\phantom{s}
 \thispagestyle{empty}


\begin{thebibliography}{99}
\bibitem{Abe}
K. Abe and K. Choi,  Stability of Lamb dipoles, \textit{Arch. Ration. Mech. Anal.}, 244(2022), 877--917.
\bibitem{A1}
V. I. Arnold,  Conditions for nonlinear stability plane curvilinear flow of an idea fluid,
\textit{Sov. Math. Dokl.}, 6(1965), 773--777.


\bibitem{A2}
V. I. Arnold,  On an a priori estimate in the theory of hydrodynamical stability, \textit{Amer. Math. Soc. Transl.}, 79(1969), 267--269.

\bibitem{AK}
V. I. Arnold and B. A. Khesin, Topological methods in hydrodynamics, 2nd ed.,  Applied Mathematical Sciences 125,  Springer, Cham,  2021.

\bibitem{BG}
C. Bardos, Y. Guo, W. Strauss,  Stable and unstable ideal plane flows, \textit{ Chinese Ann. Math. Ser. B}, 23(2002),149--164.


\bibitem{BFY}
L. Belenkaya, S. Friedlander and V. Yudovich,
The unstable spectrum of oscillating shear flows, \textit{
SIAM J. Appl. Math.,} 59(1999), 1701--1715.

\bibitem{B1}
G. R. Burton, Rearrangements of functions, maximization of convex functionals, and vortex rings, \textit{Math. Ann.}, 276(1987), 225--253.


\bibitem{B5}
G. R. Burton, Global nonlinear stability for steady ideal fluid flow in bounded planar domains,
\textit{Arch. Ration. Mech. Anal.}, 176(2005), 149--163.



\bibitem{Bjde}
G. R. Burton, Compactness and stability for planar vortex-pairs with prescribed impulse, \textit{J. Differential Equations}, 270(2021), 547--572.



\bibitem{Bcmp}
G. R. Burton, H. J. Nussenzveig Lopes and M. C. Lopes Filho, Nonlinear stability for steady vortex pairs, \textit{Comm. Math. Phys.}, 324(2013), 445--463.

\bibitem{BN}
P. Butta and P. Negrini,   On the stability problem of stationary solutions for the Euler equation on a 2-dimensional torus, \textit{ Regul. Chaotic Dyn.}, 15(2010), 637--645.

\bibitem{CZ}
A. P. Calder\'on, A. P and A. Zygmund,  Singular integrals and periodic functions, \textit{Studia Math.}, 14(1954), 249--271.


\bibitem{CWCV}
D. Cao and G. Wang, Steady vortex patches with opposite rotation directions in a planar ideal fluid, \textit{Calc. Var. Partial Differential Equations,} 58(2019), Paper No. 75.


\bibitem{CWN}
D. Cao and G. Wang, Nonlinear stability of planar vortex patches in an ideal fluid, \textit{J. Math. Fluid Mech.}, 23(2021), Paper No. 58.



 \bibitem{CD}
 K. Choi and D.  Lim, Stability of radially symmetric, monotone vorticities of 2D Euler equations, \textit{Calc. Var. Partial Differential Equations}, 61(2022),  Paper No. 120.




 \bibitem{DW}
H. R. Dullin and J. Worthington,  
Stability results for idealized shear flows on a rectangular periodic domain, \textit{J. Math. Fluid Mech.}, 20(2018),  473--484.


\bibitem{Gra}
L. Grafakos,  Classical Fourier analysis, Third edition, \textit{Graduate Texts in Mathematics, Vol. 249.} Springer, New York (2014).


\bibitem{BR}
J. Batt and G. Rein,
A rigorous stability result for the Vlasov-Poisson system in three dimensions, \textit{
Ann. Mat. Pura Appl.}, 164(1993), 133--154.


\bibitem{FS}
S. Friedlander,  W. Strauss and M. Vishik,  Nonlinear instability in an ideal fluid, \textit{Ann. Inst. H. Poincar\'e. Anal. Non Lin\'eare.},14(1997),  187--209.

\bibitem{LY}
Y. Li,  On 2D Euler equations. I. On the energy-Casimir stabilities and the spectra for linearized 2D Euler equations, \textit{ J. Math. Phys.}, 41(2000), 728--758.

\bibitem{MB}
A. J. Majda and A. L. Bertozzi, Vorticity and incompressible flow, \textit{Cambridge Texts in Applied Mathematics, Vol. 27}, Cambridge University Press, 2002.



\bibitem{MPu}
C. Marchioro and M. Pulvirenti, Mathematical theory of incompressible noviscous fluids, Springer-Verlag, 1994.


\bibitem{MSi}
L. D. Mesalkin and J. G. Sinai,  Investigation of the stability of a stationary solution of a system of equations for the plane movement of an incompressible viscous liquid, \textit{Prikl. Mat. Meh.}  25, 1140--1143 (in Russian); translated as \textit{J. Appl. Math. Mech.} 25(1961), 1700--1705.


\bibitem{Ti}
E. C. Titchmarsh,  Eigenfunction expansions associated with second-order differential equations, Vol. 2, Oxford University Press, 1958.




 \bibitem{WGu1}
G. Wang,  Nonlinear stability of planar steady Euler flows associated with semistable solutions of elliptic problems, \textit{Trans. Amer. Math. Soc.}, 375(2022),  5071--5095.

 \bibitem{WGu2}
G. Wang, Stability of 2D steady Euler flows related to least energy solutions of the Lane-Emden equation, arXiv:2104.12406.


\bibitem{WS}
 D. Wirosoetisno and T. G. Shepherd,  Nonlinear stability of Euler flows in two-dimensional periodic domains, \textit{ Geophys. Astrophys. Fluid Dynam.}, 90(1999),  229--246.

\bibitem{WG0}
G. Wolansky, M. Ghil, An extension of Arnold's second stability theorem for the Euler equations, \textit{Phys. D}, 94(1996), 161--167.


\bibitem{WG}
G. Wolansky and M. Ghil,
Nonlinear stability for saddle solutions of ideal flows and symmetry breaking. \textit{Comm. Math. Phys.}, 193(1998), 713--736.




\end{thebibliography}
\end{document}